\documentclass[11pt,reqno]{amsart}
\usepackage{indentfirst, amssymb,amsmath,amsthm, mathrsfs,cite,graphicx,float}
\newtheorem*{theoA}{Theorem A}
\newtheorem*{theoB}{Theorem B}
\newtheorem*{theoC}{Theorem C}
\newtheorem*{theoD}{Theorem D}
\newtheorem*{theoE}{Theorem E}
\newtheorem*{theoF}{Theorem F}

\newtheorem{theo}{Theorem}[section]

\newtheorem*{lemA}{\textrm{Lemma} A}

\newtheorem{cor}{Corollary}[section]

\newtheorem{defi}{Definition}

\newtheorem{prob}{Problem}[section]
\newtheorem{open problem}{Open problem}[section]
\newcommand{\pa}{\partial}
\newcommand{\ol}{\overline}
\newcommand{\be}{\begin{equation}}
\newcommand{\ee}{\end{equation}}
\newcommand{\bs}{\begin{small}}
\newcommand{\es}{\end{small}}
\newcommand{\beas}{\begin{eqnarray*}}
\newcommand{\eeas}{\end{eqnarray*}}
\newcommand{\bea}{\begin{eqnarray}}
\newcommand{\eea}{\end{eqnarray}}
\newcommand{\D}{\mathbb{D}}

\renewcommand{\epsilon}{\varepsilon}
\numberwithin{equation}{section}

\begin{document}
\title[On the univalence criteria]{On the univalence criteria for elliptic polyharmonic and polyelliptic-harmonic mappings}
\author[R. Mandal and S. K. Guin]{ Rajib Mandal and Sudip Kumar Guin}
\date{}
\address{Rajib Mandal, Department of Mathematics, Raiganj University, Raiganj, West Bengal-733134, India.}
\email{rajibmathresearch@gmail.com}
\address{Sudip Kumar Guin, Department of Mathematics, Raiganj University, Raiganj, West Bengal-733134, India.}
\email{sudipguin20@gmail.com}
\maketitle
\let\thefootnote\relax
\begin{abstract}
In this paper, we first establish Landau-Bloch-type theorems for poly $(K,K')$-elliptic harmonic mappings, which are sharp in some given cases. Thereafter, we provide several coefficient bounds for $(K,K')$-elliptic and $K$-quasiregular polyharmonic mappings with bounded minimum distortion. Furthermore, using these coefficient bounds, we establish Landau-Bloch-type theorems for these mappings.
\end{abstract}
\section{Introduction and Preliminaries}
Let $\D_r=\{z\in\mathbb{C}:|z|<r\}$ be the open disk with center at the origin and radius $r>0$ and $\D:=\D_1$. In a domain $\Omega\subset\mathbb{C}$, a $2p$-times $(p\geq 1)$ continuously differentiable complex-valued mapping $F$ is called polyharmonic if it satisfies the polyharmonic equation $\Delta^pF=\Delta(\Delta^{p-1}F)=0$, where $\Delta $ stands for the Laplacian operator 
\beas \Delta :=4\frac{\pa^2}{\pa z \pa\bar{z}}=\frac{\pa^2}{\pa x^2}+\frac{\pa^2}{\pa y^2}.\eeas
The case $p=1$ (or $p=2$) corresponds to harmonic (or biharmonic) mapping. For a simply connected domain $\Omega\subset\mathbb{C}$, $F$ is polyharmonic in $\Omega$ if, and only, if $F$ has the form 
\beas F(z)=\sum\limits_{k=1}^p|z|^{2(k-1)}G_{p-k+1}(z),\;\;z\in\Omega,\eeas
where $G_{p-k+1}(z)$ is a complex-valued harmonic mapping in $\Omega$ for each $k\in\left\{1,2,\dots,p\right\}$ (see \cite{CPW2011,LL2014}).\\[2mm]
\indent For a continuously differentiable mapping $f$ in $\D$, we use the following standard notations:
\beas \Lambda_f(z) &=&\max\limits_{0\leq t \leq 2\pi} \left|f_z(z)+e^{-2it}f_{\ol{z}}(z)\right|=|f_z(z)|+|f_{\ol{z}}(z)|\eeas
and
\beas \lambda_f(z) &=&\min\limits_{0\leq t \leq 2\pi} \left|f_z(z)+e^{-2it}f_{\ol{z}}(z)\right|=\left||f_z(z)|-|f_{\ol{z}}(z)|\right|\eeas
for maximum and minimum distortions, respectively.
For a sense-preserving harmonic mapping $f$, the Jacobian $J_f=|f_z|^2-|f_{\ol{z}}|^2>0$, and thus $J_f(z)=\Lambda_f(z)\lambda_f(z)$. In $1936$, Hans Lewy \cite{L1936} showed that a harmonic function $f$ is locally univalent (or one-to-one) at $z$ if, and only if,  $J_f(z)\ne 0$. Each harmonic mapping $f$ in $\D$ has the canonical representation $f=h+\ol{g}$, where $h$ and $g$ are analytic in $\D$ such that $g(0)=0$ (see \cite{D2004}). Harmonic mappings are currently of significant interest due to their applicability in fluid flows problems (see \cite{AC2012,CM2017}).\\[2mm]
\indent
Now we recall the classical Landau theorem (see \cite{L1926}): if $f$ is a normalized analytic function in $\D$ with $|f(z)| < M $ for $z\in \D$, then
$f$ is univalent in $\D_{\rho}$ with $\rho=1/(M+\sqrt{M^2-1})$, and $f(\D_{\rho})$ contains a disk of radius $M\rho^2$. The result is sharp, with the extremal function $Mz(1-Mz)/(M-z)$.\\
Furthermore, the classical Bloch theorem asserts the existence of a positive constant $b$ such that if $f$ is analytic on $\D$ with $f'(0)=1$, then the image $f(\D)$ contains a schlicht disk of radius $b$, {\it i.e.}, a disk of radius $b$, which is the univalent image of some region in $\D$. The Bloch constant is defined as the supremum of such constants $b$ (see \cite{CGH2000,GK2003}).\\[2mm]
\indent For a complex-valued function $f(z)=u+iv$, $z=x+iy$, the formal derivative $\D_f$ of $f$ is given by
\beas \D_f=\left(\begin{array}{ccc}u_x&u_y\\v_x&v_y\end{array}\right) \eeas
and the norm of $\D_f$ is defined by
\beas ||\D_f||=\sup_{|z|=1}|\D_fz|&=&\sup_{x^2+y^2=1}\left|(xu_x+yu_y)+i(xv_x+yv_y)\right|\\[2mm]&=&\sup_{|z|=1}|zf_z+\ol{z}f_{\ol{z}}|=|f_z|+|f_{\ol{z}}|=\Lambda_f.\eeas
\begin{defi}\cite{R1993}
In a domain $\Omega\subset \mathbb{C}$, a mapping $f:\Omega\to \mathbb{C}$ is said to be absolutely continuous on lines (or $ACL$), if for every closed rectangle $R\subset\Omega$ with sides parallel to the axes $x$ and $y$, $f$ is absolutely continuous on almost every horizontal line and almost every vertical line in $R$. Such a mapping $f$ has partial derivatives $f_x$ and $f_y$ a.e. in $\Omega$. Furthermore, we say $f\in ACL^2$ if $f\in ACL$ and its partial derivatives are locally $L^2$ integrable in $\Omega$.
\end{defi}
\begin{defi}\cite{FS1958}
A sense-preserving and continuous mapping $f$ of $\D$ onto $\mathbb{C}$ is said to be $(K,K')$-elliptic mapping if
\begin{enumerate}
\item[(1)] $f$ is $ACL^2$ in $\D$, $J_f\neq 0$ {\it a.e.} in $\D$,
\item[(2)] there exist $K\geq 1$ and $K'\geq 0$ such that
\end{enumerate}
\beas ||D_f||^2\leq KJ_f+K'\; \text{\it a.e.}\; \text{in}\;\D.\eeas
In particular, if $K'=0$, then a $(K,K')$-elliptic mapping becomes a $K$-quasiregular mapping. Therefore, a sense-preserving harmonic mapping $f$ is said to be $K$-quasiregular harmonic $(K\geq 1)$ on $\D$ if $\Lambda_f(z)\leq K\lambda_f(z)$ for all $z\in\D$.
\end{defi}
\begin{defi}
For a simply connected domain $\Omega\subset\mathbb{C}$, $F$ is a poly $(K, K')$-elliptic harmonic mapping in $\Omega$ if, and only, if $F$ has the form 
\beas F(z)=\sum\limits_{k=1}^p|z|^{2(k-1)}G_{p-k+1}(z),\;\;z\in\Omega,\eeas
where $G_{p-k+1}(z)$ is a complex-valued $(K, K')$-elliptic harmonic mapping in $\Omega$ for each $k\in\left\{1,2,\dots,p\right\}$.
\end{defi}
\indent Several authors have considered the Landau-type theorems for polyharmonic mappings. For biharmonic mappings, the Landau-type results were first studied by Abdulhadi, Muhanna and Khuri \cite{AMK2006}. For more in-depth study on Landau-type theorems for harmonic mappings, biharmonic mappings, and polyharmonic mappings, we refer to \cite{BL2019,CPW2011,CRW2014,CS1984,D2004,FL2024,L2008,LL2014,LL2019,LL2023,WPL2024}. Recently, in \cite{AK2024,LX2024}, authors considered Landau-type theorems for elliptic harmonic mappings.
In $2024$, Fu and Luo \cite{FL2024} established the following two Landau-type theorems for polyharmonic mappings.
\begin{theoA}\cite[Theorem 3.4]{FL2024}\label{tha}
Let $F(z)=\sum_{k=1}^p |z|^{2(k-1)}G_{p-k+1}(z)$ be a polyharmonic mapping in $\D$ such that $F(0)=0$, and satisfying the following conditions:
\begin{itemize}
\item[(i)] $G_{p-k+1}(z)$ is harmonic in $\D$ and $\lambda_{G_{p-k+1}}(0)-1=G_{p-k+1}(0)=0$ for $k\in\{1,\cdots, p\}$;
\item[(ii)] $|G_{p-k+1}(z)|\leq M_{p-k+1}$ for $k\in\{2,\cdots, p\}$ and $\Lambda_{G_p}(z)\leq \Lambda_p$ for all $z\in\D$.
\end{itemize}
Then $\Lambda_p\geq 1$ and $M_{p-k+1}\geq 1$ for $k\in\{2,\cdots ,p\}$. Further,  $F(z)$ is univalent in $\D_{\rho_1}$, where $\rho_1$ is the unique root in $(0,1)$ of the equation
\beas \frac{\Lambda_p(1-\Lambda_p r)}{\Lambda_p-r}-\phi(r)=0,\eeas
where
\bea\label{1z1} \phi(r)&=&\sum_{k=2}^p r^{2(k-1)}\left((2k-1)K_1(M_{p-k+1})\right.\nonumber\\[2mm]&&\left.+\sqrt{2M_{p-k+1}^2-2}\left(\frac{2(k-1)r}{\sqrt{1-r^2}}+\frac{r\sqrt{4-3r^2+r^4}}{(1-r^2)^{3/2}}\right)\right)\eea
with
\beas K_1(M_{p-k+1})=\min\left\{\sqrt{2M_{p-k+1}^2-1}\frac{4M_{p-k+1}}{\pi}\right\}.\eeas
Furthermore, $F(\D_{\rho_1})$ contains a schlicht disk with radius
{\small \beas \sigma_1=\Lambda_p^2\rho_1+(\Lambda_p^3-\Lambda_p)\log{(1-\frac{\rho_1}{\Lambda_p})}-\sum_{k=2}^p{\rho_1^{2k-1}\left(K_1(M_{p-k+1})+\sqrt{2M_{p-k+1}^2-2}\frac{\rho_1}{\sqrt{1-\rho_1^2}}\right)}.\eeas}
When $M_{p-k+1}=1$, $k=2,\cdots,p$, the result is sharp, with an extremal function given by
{\small \beas F_1(z)=\Lambda_p\int_{0}^z{\frac{\frac{1}{\Lambda_p}-\zeta}{1-\frac{\zeta}{\Lambda_p}}d\zeta}-\sum_{k=2}^p{|z|^{2(k-1)}z}=\Lambda_p^2z+(\Lambda_p^3-\Lambda_p)\log{(1-\frac{z}{\Lambda_p})}-\sum_{k=2}^p|z|^{2(k-1)}z.\eeas}
\end{theoA}
\begin{theoB}\cite[Theorem 3.5]{FL2024}\label{thb}
Let $F(z)=\sum_{k=1}^p |z|^{2(k-1)}G_{p-k+1}(z)$ be a polyharmonic mapping in $\D$ such that $F(0)=\lambda_F(0)-1=0$, and satisfying the following conditions:
\begin{itemize}
\item[(i)] $G_{p-k+1}(z)$ is harmonic in $\D$ and $G_{p-k+1}(0)=0$ for $k\in\{1,\cdots, p\}$; 
\item[(ii)] $\Lambda_{G_{p-k+1}}(z)\leq \Lambda_{p-k+1}$ for $k\in\{2,\cdots, p\}$ and $|G_{p}(z)|\leq M_{p}$ for all $z\in\D$.
\end{itemize}
Then $\Lambda_{p-k+1}\geq 0$ and $M_{p}\geq 1$ for $k\in\{2,\cdots ,p\}$. Further, $F(z)$ is univalent in $\D_{\rho_2}$, where $\rho_2$ is the unique root in $(0,1)$ of the equation
\beas1-\sqrt{2M_p^2-2}\cdot \frac{r\sqrt{r^4-3r^2+4}}{(1-r^2)^{3/2}}-\sum_{k=2}^p{(2k-1)}\Lambda_{p-k+1}r^{2(k-1)}=0.\eeas
Furthermore, $F(\D_{\rho_2})$ contains a schlicht disk with radius
\beas \sigma_2=\rho_2-\sqrt{2M_p^2-2}\cdot \frac{\rho_2^2}{\sqrt{1-\rho_2^2}}-\sum_{k=2}^p{\rho_2^{2k-1}}\Lambda_{p-k+1}.
\eeas
When $M_{p}=1$, the result is sharp, with an extremal function given by
$ F_2(z)=z-\sum_{k=1}^{p-1}\Lambda_{p-k}|z|^{2k}z $.
\end{theoB}
In \cite{CRW2014}, Chen {\it et al.} established the following Landau-type theorem for bounded polyharmonic mappings under a suitable restriction.
\begin{theoC}\cite[Theorem 2]{CRW2014}\label{thd}
Let $F$ be a polyharmonic mapping of the form
{\small \bea\label{1a1} F(z)=a_0+\sum_{k=1}^p |z|^{2(k-1)}\left(h_k(z)+\ol{g_k(z)}\right)=a_0+\sum_{k=1}^p |z|^{2(k-1)}\sum_{n=1}^\infty \left(a_{n,k}z^n+\ol{b_{n,k}z^n}\right)\eea}
and all its non-zero coefficients $a_{n,k_1}, a_{n,k_2}$ and $b_{n,k_3},b_{n,k_4}$ satisfy the conditions
\bea\label{1a2} \left|\arg{\frac{a_{n,k_1}}{a_{n,k_2}}}\right|\leq \frac{\pi}{2} \quad \text{and}\quad  \left|\arg{\frac{b_{n,k_3}}{a_{n,k_4}}}\right|\leq \frac{\pi}{2}.\eea
If $|F(z)|\leq M$ in $\D$ for some $M>1$ and $F(0)=0=J_F(0)-1$, then $F$ is univalent in $\D_{\rho_3}$, where $\rho_3$ is the least positive root of the equation
\beas 1-\sqrt{M^4-1}\left(\frac{2r-r^2}{(1-r)^2}+\sum_{k=1}^{p-1} \frac{r^{2k}}{(1-r)^2}+2\sum_{k=1}^{p-1}\frac{kr^{2k}}{1-r}\right)=0.\eeas
Furthermore, $F(\D_{\rho_3})$ contains a schlicht disk with radius
\beas \sigma_3=\lambda_0(M) \rho_3\left(1-\sqrt{M^4-1}\frac{\rho_3}{1-\rho_3}-\sqrt{M^4-1}\sum_{k=1}^{p-1}\frac{2\rho_3^{2k}}{1-\rho_3}\right), \eeas
where
\beas \lambda_0(M) =\left\{\begin{array}{lll}\frac{\sqrt{2}}{\sqrt{M^2-1}+\sqrt{M^2+1}}, &1\leq M\leq M_0=\frac{\pi}{2\sqrt[4]{2\pi^2-16}}\approx 1.1296,\\\frac{\pi}{4M}, &M>M_0.\end{array}\right.\eeas
\end{theoC}
In $2025$, Wang {\it et al.} \cite{WPL2024} established the following version of Landau-type theorem for polyharmonic mappings.
\begin{theoD}\cite[Theorem 3.1]{WPL2024}\label{the}
Let $F$ be a polyharmonic mapping of the form (\ref{1a1}) and all its non-zero coefficients $a_{n,k_1}, a_{n,k_2}$ and $b_{n,k_3},b_{n,k_4}$ satisfy the conditions (\ref{1a2}). If $\Lambda_F(z)\leq M$ in $\D$ for some $M>1$ and $F(0)=0=J_F(0)-1$, then $F$ is univalent in the disk $\D_{\rho_4}$, where $\rho_4$ is the least positive root of the equation
\beas &&1-\sqrt{M^4-1}\left(\frac{r}{1-r}+\sum_{k=1}^{p-1}r^{2k}\left(\frac{1}{\sqrt{5}}+\frac{2r-r^2}{\sqrt{10}(1-r)^2}\right)\right)\\[2mm]&&-2\sqrt{M^4-1}\sum_{k=1}^{p-1}kr^{2(k-1)}\left(\frac{r^2}{\sqrt{5}}+\frac{r^3}{\sqrt{10}(1-r)}\right)=0.\eeas
Furthermore, $F(\D_{\rho_4})$ contains a schlicht disk with radius
{\small \beas \sigma_4=\lambda_1(M)\rho_4\left(1+\sqrt{M^4-1}\left(\frac{\rho_4+\ln{(1-\rho_4)}}{\rho_4}-\sum_{k=1}^{p-1}\rho_4^{2k}\left(\frac{1}{\sqrt{5}}+\frac{\rho_4}{\sqrt{10}(1-\rho_4)}\right)\right)\right),\eeas}
where
\beas \lambda_1(M)=\frac{\sqrt{2}}{\sqrt{M^2-1}+\sqrt{M^2+1}}.\eeas
\end{theoD}
In $2024$, Allu and Kumar \cite{AK2024} established the following two Landau-type theorems for $(K,K')$-elliptic and $K$-quasiregular harmonic mappings.
\begin{theoE}\cite[Theorem 2.2]{AK2024}\label{thf}
Let $f=h+\ol{g}$ be a $(K,K')$-elliptic harmonic mapping in $\D$ such that $f(0)=0, \lambda_f(0)=1$ and $\lambda_f(z)\leq \lambda$ for all $z\in\D$. Then $f$ is univalent in $\D_{\rho_5}$ with
\beas \rho_5=\rho_5(\lambda)=\frac{1}{1+K\lambda+\sqrt{K'}}\eeas
and $f(\D_{\rho_5})$ contains a schlicht disk with radius
\beas \sigma_5=\rho_5(\lambda)+(K\lambda+\sqrt{K'})\left(\rho_5(\lambda)+\ln{\left(\left(K\lambda+\sqrt{K'}\right)\rho_5(\lambda)\right)}\right).\eeas
\end{theoE}
\begin{theoF}\cite[Theorem 2.3]{AK2024}\label{thg}
Let $f=h+\ol{g}$ be a harmonic $K$-quasiregular mapping defined in $\D$ such that $f(0)=0, J_f(0)=1$ and $\lambda_f(z)\leq \lambda$ for all $z\in\D$. Then $f$ is univalent in $\D_{\rho_6}$ with
\beas \rho_6=\rho_6(\lambda)=\frac{1}{1+\lambda K^{3/2}}\eeas
and $f(\D_{\rho_6})$ contains a schlicht disk with radius
\beas \sigma_6=\frac{\rho_6(\lambda)}{\sqrt{K}}+K\lambda\left(\rho_6(\lambda)+\ln{\left(\lambda K^{3/2}\rho_6(\lambda)\right)}\right).\eeas
\end{theoF}
Some improved results of Theorem E and Theorem F are given in \cite{LX2024}.\\
The subsequent lemma is pivotal in substantiating our principal findings.
\begin{lemA}\label{la1}\cite{CP2020}
Let $f\in C^1$ be a sense-preserving mapping. Then $f$ is an elliptic mapping if, and only if, there exist constants $c\in[0,1)$ and $d\in[0,\infty)$ such that
\beas |f_{\ol{z}}(z)|\leq c|f_z(z)|+d \quad\text{for}\quad z\in \D.\eeas
In particular, if $f$ is a $(K,K')$-elliptic mapping, then
\beas c=\frac{K-1}{K+1},\;d=\frac{\sqrt{K'}}{1+K}.\eeas
\end{lemA}
It is natural to raise the following problems.
\begin{prob}\label{p1}
Can we establish the Landau-type theorems for $(K,K')$-elliptic polyharmonic, $K$-quasiregular polyharmonic and poly $(K,K')$-elliptic harmonic mappings? Furthermore, can we establish a few sharp results?
\end{prob}
\indent To present an affirmative answer to Problem \ref{p1}, this paper is organized into the following sections. In section 2, we present statements of Landau-type theorems for poly $(K,K')$-elliptic harmonic mappings, which are sharp in some given cases. Furthermore,  statements of theorems for coefficient estimates and Landau-type theorems for $(K,K')$-elliptic polyharmonic and $K$-quasiregular polyharmonic mappings are also provided. In section 3, we provide the proofs of the main results.
\section{Main Results}
We first establish two Landau-type theorems for poly $(K,K')$-elliptic harmonic mappings that correspond to Theorem A and Theorem B as follows:
\begin{theo}\label{tha1}
Let $F(z)=\sum_{k=1}^p |z|^{2(k-1)}G_{p-k+1}(z)$ be a poly $(K,K')$-elliptic harmonic mapping in $\D$ such that $F(0)=0$, and satisfying the following conditions:
\begin{itemize}
\item[(i)] $G_{p-k+1}(z)$ is elliptic harmonic in $\D$ and $\lambda_{G_{p-k+1}}(0)-1=G_{p-k+1}(0)=0$ for $k\in\{1,\cdots, p\}$;
\item[(ii)] $|G_{p-k+1}(z)|\leq M_{p-k+1}$ for $k\in\{2,\cdots, p\}$ and $\lambda_{G_p}(z)\leq \Lambda_p$ for all $z\in\D$.
\end{itemize}
Then $F(z)$ is univalent in $\D_{r_1}$, where $r_1$ is the unique root in $(0,1)$ of the equation
\bea\label{2z1} \frac{\Lambda_p'(1-\Lambda_p' r)}{\Lambda_p'-r}-\phi(r)=0,\eea
where $\phi(r)$ is given by (\ref{1z1}) and $\Lambda_p'=\left(K\Lambda_p+\sqrt{K^2\Lambda_p^2+4K'}\right)/2$.\\[2mm]
Furthermore, $F(\D_{r_1})$ contains a schlicht disk with radius
{\small \beas\label{2z2} R_1=\Lambda_p'^2r_1+(\Lambda_p'^3-\Lambda_p')\log{(1-\frac{r_1}{\Lambda_p'})}-\sum_{k=2}^p{r_1^{2k-1}\left(K_1(M_{p-k+1})+\sqrt{2M_{p-k+1}^2-2}\frac{r_1}{\sqrt{1-r_1^2}}\right)}.\eeas}
When $K=1,K'=0$ and $M_{p-k+1}=1$ for $k=2,\cdots,p$, the result is sharp.
\end{theo}
\begin{theo}\label{tha2}
Let $F(z)=\sum_{k=1}^p |z|^{2(k-1)}G_{p-k+1}(z)$ be a poly $(K,K')$-elliptic harmonic mapping in $\D$ such that $F(0)=\lambda_F(0)-1=0$, and satisfying the following conditions:
\begin{itemize}
\item[(i)] $G_{p-k+1}(z)$ is elliptic harmonic in $\D$ and $G_{p-k+1}(0)=0$ for $k\in\{1,\cdots, p\}$; 
\item[(ii)] $\lambda_{G_{p-k+1}}(z)\leq \Lambda_{p-k+1}$ for $k\in\{2,\cdots, p\}$ and $|G_{p}(z)|\leq M_{p}$ for all $z\in\D$.
\end{itemize}
Then $F(z)$ is univalent in $\D_{r_2}$, where $r_2$ is the unique root in $(0,1)$ of the following equation
\bea\label{2z3} 1-\sqrt{2M_p^2-2}\cdot \frac{r\sqrt{r^4-3r^2+4}}{(1-r^2)^{3/2}}-\sum_{k=2}^p{(2k-1)}\Lambda_{p-k+1}'r^{2(k-1)}=0\eea
with $\Lambda_{p-k+1}'=\left(K\Lambda_{p-k+1}+\sqrt{K^2\Lambda_{p-k+1}^2+4K'}\right)/2$.\\[2mm]
Furthermore, $F(\D_{r_2})$ contains a schlicht disk with radius
\bea\label{2z4} R_2=r_2-\sqrt{2M_p^2-2}\cdot \frac{r_2^2}{\sqrt{1-r_2^2}}-\sum_{k=2}^p{r_2^{2k-1}}\Lambda_{p-k+1}'.
\eea
When $K=1,K'=0$ and $M_{p}=1$, the result is sharp.
\end{theo}
Next, we establish coefficient estimates for $(K, K')$-elliptic and $K$-quasiregular polyharmonic mappings with bounded minimum distortion.
\begin{theo}\label{la3}
Let $F$ be a $(K, K')$-elliptic polyharmonic mapping of the form (\ref{1a1}) satisfying (\ref{1a2})
with $\lambda_F(z)\leq \lambda$ for $z\in\D$. Then we have
\beas \sum_{k=1}^p \sum_{n=1}^\infty \left((n+k-1)^2+(k-1)^2\right)\left(|a_{n,k}|^2+|b_{n,k}|^2\right)\leq \frac{(K^2+1)\lambda^2+2K\sqrt{K'}\lambda+K'}{2}.\eeas
In particular, we have
\bea\label{2a1}\left\{\begin{array}{lll}|a_{n,k}|+|b_{n,k}|\leq \frac{\sqrt{(K^2+1)\lambda^2+2K\sqrt{K'}\lambda+K'}}{n}\;\;\text{for all}\;\;n\geq2,k=1,\\[2mm]
|a_{n,k}|+|b_{n,k}|\leq \frac{\sqrt{(K^2+1)\lambda^2+2K\sqrt{K'}\lambda+K'}}{\sqrt{10}} \;\;\text{for all}\;\; n\geq 2, 2\leq k\leq p,\\[2mm]
|a_{n,k}|+|b_{n,k}|\leq \frac{\sqrt{(K^2+1)\lambda^2+2K\sqrt{K'}\lambda+K'}}{\sqrt{5}}\;\;\text{for all}\;\;n=1,2\leq k\leq p.\end{array}\right.\eea
\end{theo}
\begin{cor}\label{c1}
Let $F$ be a $K$-quasiregular polyharmonic mapping satisfying the hypothesis of \textrm{Theorem \ref{la3}}. Then we have
\beas\left\{\begin{array}{lll}|a_{n,k}|+|b_{n,k}| \leq \frac{\lambda\sqrt{K^2+1}}{n}\;\;\text{for all}\;\;n\geq2,k=1,\\[2mm]
|a_{n,k}|+|b_{n,k}|\leq \frac{\lambda\sqrt{K^2+1}}{\sqrt{10}} \;\;\text{for all}\;\; n\geq 2, 2\leq k\leq p,\\[2mm]
|a_{n,k}|+|b_{n,k}|\leq \frac{\lambda\sqrt{K^2+1}}{\sqrt{5}}\;\;\text{for all}\;\;n=1,2\leq k\leq p.\end{array}\right.\eeas
\end{cor}
\begin{theo}\label{la4}
Let $F$ be a $(K, K')$-elliptic polyharmonic mapping satisfying the hypothesis of \textrm{Theorem \ref{la3}} with $F(0)=\lambda_F(0)-1=0$. Then we have
\beas\left\{\begin{array}{lll}|a_{n,k}|+|b_{n,k}|\leq \frac{\sqrt{(K^2+1)\lambda^2+2K\sqrt{K'}\lambda+K'-1}}{n}\;\;\text{for all}\;\;n\geq2,k=1,\\[2mm]
|a_{n,k}|+|b_{n,k}|\leq \frac{\sqrt{(K^2+1)\lambda^2+2K\sqrt{K'}\lambda+K'-1}}{\sqrt{10}} \;\;\text{for all}\;\; n\geq 2, 2\leq k\leq p,\\[2mm]
|a_{n,k}|+|b_{n,k}|\leq \frac{\sqrt{(K^2+1)\lambda^2+2K\sqrt{K'}\lambda+K'-1}}{\sqrt{5}}\;\;\text{for all}\;\;n=1,2\leq k\leq p.\end{array}\right.\eeas
\end{theo}
\begin{cor}\label{c2}
Let $F$ be a $K$-quasiregular polyharmonic mapping satisfying the hypothesis of \textrm{Theorem \ref{la4}}. Then we have
\beas\left\{\begin{array}{lll}|a_{n,k}|+|b_{n,k}|\leq \frac{\sqrt{(K^2+1)\lambda^2-1}}{n}\;\;\text{for all}\;\;n\geq2,k=1,\\[2mm]
|a_{n,k}|+|b_{n,k}|\leq \frac{\sqrt{(K^2+1)\lambda^2-1}}{\sqrt{10}} \;\;\text{for all}\;\; n\geq 2, 2\leq k\leq p,\\[2mm]
|a_{n,k}|+|b_{n,k}|\leq \frac{\sqrt{(K^2+1)\lambda^2-1}}{\sqrt{5}}\;\;\text{for all}\;\;n=1,2\leq k\leq p.\end{array}\right.\eeas
\end{cor}
\begin{theo}\label{la5}
Let $F$ be a $(K, K')$-elliptic polyharmonic mapping satisfying the hypothesis of \textrm{Theorem \ref{la3}} with $F(0)=J_F(0)-1=0$. Then we have
\beas\left\{\begin{array}{lll}|a_{n,k}|+|b_{n,k}|\leq \frac{\sqrt{(K^2+1)\lambda^2+2K\sqrt{K'}\lambda+K'-\frac{1}{K+K'}}}{n}\;\;\text{for all}\;\;n\geq2,k=1,\\[2mm]
|a_{n,k}|+|b_{n,k}|\leq \frac{\sqrt{(K^2+1)\lambda^2+2K\sqrt{K'}\lambda+K'-\frac{1}{K+K'}}}{\sqrt{10}} \;\;\text{for all}\;\; n\geq 2, 2\leq k\leq p,\\[2mm]
|a_{n,k}|+|b_{n,k}|\leq \frac{\sqrt{(K^2+1)\lambda^2+2K\sqrt{K'}\lambda+K'-\frac{1}{K+K'}}}{\sqrt{5}}\;\;\text{for all}\;\;n=1,2\leq k\leq p.\end{array}\right.\eeas
\end{theo}
\begin{cor}\label{c3}
Let $F$ be a $K$-quasiregular polyharmonic mapping satisfying the hypothesis of \textrm{Theorem \ref{la5}}. Then we have
\beas\left\{\begin{array}{lll}|a_{n,k}|+|b_{n,k}|\leq \frac{\sqrt{(K^2+1)\lambda^2-1/K}}{n}\;\;\text{for all}\;\;n\geq2,k=1,\\[2mm]
|a_{n,k}|+|b_{n,k}|\leq \frac{\sqrt{(K^2+1)\lambda^2-1/K}}{\sqrt{10}} \;\;\text{for all}\;\; n\geq 2, 2\leq k\leq p,\\[2mm]
|a_{n,k}|+|b_{n,k}|\leq \frac{\sqrt{(K^2+1)\lambda^2-1/K}}{\sqrt{5}}\;\;\text{for all}\;\;n=1,2\leq k\leq p.\end{array}\right.\eeas
\end{cor}
In the following, we establish new version of Landau-type theorems for $(K, K')$-elliptic and $K$-quasiregular polyharmonic mappings.
\begin{theo}\label{th1}
Let $F$ be a $(K, K')$-elliptic polyharmonic mapping defined on the unit disk $\D$ of the form
\bea\label{3a1} F(z)&=&a_0+\sum_{k=1}^p |z|^{2(k-1)}\left(h_k(z)+\ol{g_k(z)}\right)=\sum_{k=1}^p|z|^{2(k-1)}G_{p-k+1}(z)\eea
with
\beas G_{p-k+1}(z)&=&h_k(z)+\ol{g_k(z)}=\sum_{n=1}^\infty a_{n,k}z^n+\ol{\sum_{n=1}^\infty b_{n,k}z^n}\;\;\text{for}\;\;k=2,3,\dots ,p\\\text{and}\;\;G_p(z)&=&a_0+\sum_{n=1}^\infty a_{n,1}z^n+\ol{\sum_{n=1}^\infty b_{n,1}z^n}\eeas
such that $F(0)=0$, $\lambda_F(0)=1$ and $\lambda_F(z)\leq \lambda$ in $\D$ for some $\lambda$ satisfying $(K^2+1)\lambda^2+2K\sqrt{K'}\lambda+K'>1$. Then $F$ is univalent on the disk $\D_{r_3}$, where $r_3$ is the unique root in $(0,1)$ of the equation
{\small \beas &&1-\sqrt{(K^2+1)\lambda^2+2K\sqrt{K'}\lambda+K'-1}\Bigg( \frac{r}{1-r}\nonumber\\[2mm]
&&\left.+\sum_{k=2}^pr^{2(k-1)}\left(\frac{1}{\sqrt{5}}+\frac{1}{\sqrt{10}}\frac{2r-r^2}{(1-r)^2}\right)+2\sum_{k=2}^p (k-1)r^{2(k-1)}\left(\frac{1}{\sqrt{5}}+\frac{r}{\sqrt{10}(1-r)}\right)
\right)=0.\eeas}
Furthermore, $F(\D_{r_3})$ contains a schlicht disk with radius
\beas R_3&=&r_3+\sqrt{(K^2+1)\lambda^2+2K\sqrt{K'}\lambda+K'-1}\Bigg(\ln(1-r_3)+r_3\\[2mm]
&&\left.-\sum_{k=2}^{p}r_3^{2(k-1)}\left(\frac{r_3}{\sqrt{5}}+\frac{r_3^2}{\sqrt{10}(1-r_3)}\right)\right).\eeas
\end{theo}
\begin{cor}\label{c4}
Let $F$ be a $K$-quasiregular polyharmonic mapping defined on the unit disk $\D$ of the form (\ref{3a1}) such that $F(0)=0,\lambda_F(0)=1$ and $\lambda_F(z)\leq \lambda$ in $\D$ for some $\lambda>1/\sqrt{K^2+1}$. Then $F$ is univalent on the disc $\D_{r_4}$, where $r_4$ is the unique root in $(0,1)$ of the equation
\beas &&1-\sqrt{(K^2+1)\lambda^2-1}\left(\frac{r}{1-r}+\sum_{k=2}^pr^{2(k-1)}\left(\frac{1}{\sqrt{5}}+\frac{1}{\sqrt{10}}\frac{2r-r^2}{(1-r)^2}\right)\right.\nonumber\\[2mm]
&&\left.+2\sum_{k=2}^p (k-1)r^{2(k-1)}\left(\frac{1}{\sqrt{5}}+\frac{r}{\sqrt{10}(1-r)}\right)
\right)=0.\eeas
Furthermore, $F(\D_{r_4})$ contains a schlicht disk with radius
\beas R_4&=&r_4+\sqrt{(K^2+1)\lambda^2-1}\left(\ln(1-r_4)+r_4-\sum_{k=2}^{p}r_4^{2(k-1)}\left(\frac{r_4}{\sqrt{5}}+\frac{r_4^2}{\sqrt{10}(1-r_4)}\right)\right).\eeas
\end{cor}
\begin{theo}\label{th2}
Let $F$ be a $(K, K')$-elliptic polyharmonic mapping defined on the unit disk $\D$ of the form (\ref{3a1}) such that $F(0)=0$, $J_F(0)=1$ and $\lambda_F(z)\leq \lambda$ in $\D$ for some $\lambda$ satisfying $(K^2+1)\lambda^2+2K\sqrt{K'}\lambda+K'>1/(K+K')$. Then $F$ is univalent on the disk $\D_{r_5}$, where $r_5$ is the unique root in $(0,1)$ of the equation
{\small \beas  &&\frac{1}{\sqrt{K+K'}}-\sqrt{(K^2+1)\lambda^2+2K\sqrt{K'}\lambda+K'-\frac{1}{K+K'}}\Bigg( \frac{r}{1-r}\nonumber\\[2mm]&&+\sum_{k=2}^pr^{2(k-1)}\left(\frac{1}{\sqrt{5}}+\frac{1}{\sqrt{10}}\frac{2r-r^2}{(1-r)^2}\right)+2\sum_{k=2}^p (k-1)r^{2(k-1)}\left(\frac{1}{\sqrt{5}}+\frac{r}{\sqrt{10}(1-r)}\right)\Bigg)=0.
\eeas}
Furthermore, $f(\D_{r_5})$ contains a schlicht disk with radius
{\small \beas R_5&=&\frac{r_5}{\sqrt{K+K'}}+\sqrt{(K^2+1)\lambda^2+2K\sqrt{K'}\lambda+K'-\frac{1}{K+K'}}\Bigg(\ln(1-r_5)\nonumber\\[2mm]&&\left.+r_5-\sum_{k=2}^{p}r_5^{2(k-1)}\left(\frac{r_5}{\sqrt{5}}+\frac{r_5^2}{\sqrt{10}(1-r_5)}\right)\right).\eeas}
\end{theo}
\begin{cor}\label{c5}
Let $F$ be a $K$-quasiregular polyharmonic mapping defined on the unit disk $\D$ of the form (\ref{3a1}) such that $F(0)=0$, $J_F(0)=1$ and $\lambda_F(z)\leq \lambda$ in $\D$ for some $\lambda$ satisfying $\lambda>1/\sqrt{K(K^2+1)}$. Then $F$ is univalent on the disk $\D_{r_6}$, where $r_6$ is the unique root in $(0,1)$ of the equation
{\small \beas  &&\frac{1}{\sqrt{K}}-\sqrt{(K^2+1)\lambda^2-\frac{1}{K}}\Bigg( \frac{r}{1-r}+\sum_{k=2}^pr^{2(k-1)}\left(\frac{1}{\sqrt{5}}+\frac{1}{\sqrt{10}}\frac{2r-r^2}{(1-r)^2}\right)\nonumber\\[2mm]
&&\left.+2\sum_{k=2}^p (k-1)r^{2(k-1)}\left(\frac{1}{\sqrt{5}}+\frac{r}{\sqrt{10}(1-r)}\right)
\right)=0.
\eeas}
Furthermore, $f(\D_{r_6})$ contains a schlicht disk with radius
{\small \beas R_6=\frac{r_6}{\sqrt{K}}+\sqrt{(K^2+1)\lambda^2-\frac{1}{K}}\left(\ln(1-r_6)+r_4-\sum_{k=2}^{p}r_6^{2(k-1)}\left(\frac{r_6}{\sqrt{5}}+\frac{r_6^2}{\sqrt{10}(1-r_6)}\right)\right).\eeas}
\end{cor}
\section{Proofs of the Main Results}
\begin{proof}[{\bf Proof of Theorem \ref{tha1}}]
Since $G_p(z)$ is an elliptic harmonic mapping with $\lambda_{G_p}(z)\leq \Lambda_p$, we have
\beas \Lambda_{G_p}^2(z)\leq KJ_{G_p}(z)+K'\leq K\Lambda_p\Lambda_{G_p}(z)+K'\quad\text{for}\quad z\in\D.\eeas
Hence, we have
\bea\label{3z1} \Lambda_{G_p}(z)\leq \frac{K\Lambda_p+\sqrt{K^2\Lambda_p^2+4K'}}{2}:=\Lambda_p'\quad\text{for}\quad z\in\D.\eea
Consequently, Theorem A implies that $F(z)$ is univalent in $\D_{r_1}$, where $r_1$ is the unique root in $(0,1)$ of the equation given by (\ref{2z1}), and $F(\D_{r_1})$ contains a schlicht disk $\D_{R_1}$.\\[2mm]
\indent If $K=1$ and $K'=0$, we have $\Lambda_{G_p}(z)\leq \Lambda_p$ for $z\in\D$. Thus, when $K=1$, $K'=0$ and $M_{p-k+1}=1$ for $k=2,\cdots,p$, the function
\beas F_1(z)=\Lambda_p^2z+(\Lambda_p^3-\Lambda_p)\log{(1-\frac{z}{\Lambda_p})}-\sum_{k=2}^p|z|^{2(k-1)}z\eeas
implies that the radii $r_1$ and $R_1$ are sharp.
\end{proof}
\begin{proof}[{\bf Proof of Theorem \ref{tha2}}]
In view of (\ref{3z1}), for each $k\in\{2,\cdots, p\}$ and $z\in\D$, we have
\beas \Lambda_{G_{p-k+1}}(z)\leq \frac{K\Lambda_{p-k+1}+\sqrt{K^2\Lambda_{p-k+1}^2+4K'}}{2}:=\Lambda_{p-k+1}'.\eeas
Consequently, Theorem B implies that $F(z)$ is univalent in $\D_{r_2}$, where $r_2$ is the unique root in $(0,1)$ of the equation given by (\ref{2z3}), and $F(\D_{r_2})$ contains a schlicht disk $\D_{R_2}$.\\[2mm]
\indent If $K=1$ and $K'=0$, we have $\Lambda_{G_{p-k+1}}(z)\leq \Lambda_{p-k+1}$ for each $k\in\{2,\cdots, p\}$ and $z\in\D$. Thus, when $K=1$, $K'=0$ and $M_{p}=1$, the function $F_2(z)=z-\sum_{k=2}^{p}\Lambda_{p-k+1}|z|^{2(k-1)}z $
implies that the radii $r_2$ and $R_2$ are sharp.
\end{proof}
\begin{proof}[{\bf Proof of Theorem \ref{la3}}]
Differentiating $F(z)$ partially with respect to $z$ and $\ol{z}$, respectively, we have
\beas F_z(z)&=&\sum_{k=1}^p |z|^{2(k-1)}h_k'(z)+\sum_{k=2}^p (k-1)\ol{z}|z|^{2(k-2)}\left(h_k(z)+\ol{g_k(z)}\right),\\[2mm] F_{\ol{z}}(z)&=&\sum_{k=1}^p |z|^{2(k-1)}g_k'(z)+\sum_{k=2}^p (k-1)z|z|^{2(k-2)}\left(h_k(z)+\ol{g_k(z)}\right).\eeas
Note that from the hypothesis (\ref{1a2}), we have $\textrm{Re}(a_{n,k_1}\ol{a_{n,k_2}})\geq 0$ and $\textrm{Re}(b_{n,k_1}\ol{b_{n,k_2}})\geq 0$ for each $k_1,k_2\in\{1,\cdots,p\}$ and thus, using Parseval's identity, we have
\bea\label{2a2} \frac{1}{2\pi}\int_0^{2\pi}|F_z(re^{i\theta})|^2d\theta &\geq&\sum_{k=1}^p r^{4(k-1)}\sum_{n=1}^\infty n^2|a_{n,k}|^2r^{2(n-1)}\nonumber\\[2mm]
&&+\sum_{k=2}^p(k-1)^2r^{4(k-1)}\sum_{k=1}^\infty (|a_{n,k}|^2+|b_{n,k}|^2)r^{2(n-1)}\nonumber\\[2mm]
&&+2\sum_{k=2}^p (k-1)r^{4(k-1)}\sum_{n=1}^\infty n|a_{n,k}|^2r^{2(n-1)}.\eea
Also, similarly we have
\bea\label{2a3} \frac{1}{2\pi}\int_0^{2\pi}|F_{\ol{z}}(re^{i\theta})|^2d\theta &\geq&\sum_{k=1}^p r^{4(k-1)}\sum_{n=1}^\infty n^2|b_{n,k}|^2r^{2(n-1)}\nonumber\\[2mm]
&&+\sum_{k=2}^p(k-1)^2r^{4(k-1)}\sum_{k=1}^\infty (|a_{n,k}|^2+|b_{n,k}|^2)r^{2(n-1)}\nonumber\\[2mm]
&&+2\sum_{k=2}^p (k-1)r^{4(k-1)}\sum_{n=1}^\infty n|b_{n,k}|^2r^{2(n-1)}.\eea
Since $\lambda_F(z)\leq \lambda$ and $F$ is sense preserving, we have
\beas |F_z|\leq \lambda +|F_{\ol{z}}|.\eeas
Using \textrm{Lemma A}, we have
\beas |F_z|-\lambda \leq |F_{\ol{z}}|\leq c|F_z|+d,\eeas
which shows that
\beas |F_z|\leq \frac{\lambda +d}{1-c}.\eeas
Therefore, in view of \textrm{Lemma A}, we have
\beas |F_z|^2+|F_{\ol{z}}|^2&\leq& |F_z|^2+\left(c|F_z|+d\right)^2\\[2mm]
&=& (1+c^2)|F_z|^2+2cd|F_z|+d^2\\[2mm]
&\leq & (1+c^2)\left(\frac{\lambda +d}{1-c}\right)^2+2cd\frac{\lambda +d}{1-c}+d^2\\[2mm]
&=& \frac{(K^2+1)\lambda^2+2K\sqrt{K'}\lambda+K'}{2}.\eeas
Thus, we have
{\small \bea\label{2a4} \frac{1}{2\pi}\int_0^{2\pi}\left(|F_z(re^{i\theta})|^2+|F_{\ol{z}}(re^{i\theta})|^2\right)d\theta\leq \frac{1}{2\pi}\int_0^{2\pi}\frac{(K^2+1)\lambda^2+2K\sqrt{K'}\lambda+K'}{2}d\theta\eea}
In view of (\ref{2a2}), (\ref{2a3}) and (\ref{2a4}), we have
\beas &&\sum_{k=1}^p r^{4(k-1)}\sum_{n=1}^\infty n^2(|a_{n,k}|^2+|b_{n,k}|^2)r^{2(n-1)}\\[2mm]
&&+2\sum_{k=2}^p(k-1)^2r^{4(k-1)}\sum_{k=1}^\infty (|a_{n,k}|^2+|b_{n,k}|^2)r^{2(n-1)}\\[2mm]
&&+2\sum_{k=2}^p (k-1)r^{4(k-1)}\sum_{n=1}^\infty n(|a_{n,k}|^2+|b_{n,k}|^2)r^{2(n-1)}\\[2mm]
&\leq& \frac{(K^2+1)\lambda^2+2K\sqrt{K'}\lambda+K'}{2}.\eeas
Letting $r\to 1^{-1}$ in the above inequality, we have
{\small \bea\label{2a5} \sum_{k=1}^p \sum_{n=1}^\infty \left((n+k-1)^2+(k-1)^2\right)\left(|a_{n,k}|^2+|b_{n,k}|^2\right)\leq \frac{(K^2+1)\lambda^2+2K\sqrt{K'}\lambda+K'}{2}.\eea}
From (\ref{2a5}), we have
\beas  \sum_{n=2}^\infty n^2\left(|a_{n,k}|^2+|b_{n,k}|^2\right)\leq \frac{(K^2+1)\lambda^2+2K\sqrt{K'}\lambda+K'}{2}\quad \text{for}\quad k=1,\eeas
from which, we find that
\bea\label{2a6} |a_{n,k}|^2+|b_{n,k}|^2\leq \frac{(K^2+1)\lambda^2+2K\sqrt{K'}\lambda+K'}{2n^2}\;\;\text{for all}\;\;n\geq2,k=1.\eea
Using the inequality $x^2+y^2\geq (x+y)^2/2$ in (\ref{2a6}), we have
\beas |a_{n,k}|+|b_{n,k}|\leq \frac{\sqrt{(K^2+1)\lambda^2+2K\sqrt{K'}\lambda+K'}}{n}.\eeas
Furthermore,
\bea\label{2a7} \left\{\begin{array}{lll} (n+k-1)^2+(k-1)^2\geq 10 \;\;\text{for all}\;\; n\geq 2, 2\leq k\leq p,\\[2mm]
(n+k-1)^2+(k-1)^2\geq 5\;\;\text{for all}\;\;n=1,2\leq k\leq p,\end{array}\right.\eea
and (\ref{2a5}) leads to the proof of the remaining two cases. This completes the proof.
\end{proof}
\begin{proof}[{\bf Proof of Corollary \ref{c1}}]
Since a $K$-quasiregular polyharmonic mapping is a $(K,0)$-elliptic polyharmonic mapping, by substituting $K'=0$ in \textrm{Theorem \ref{la3}}, we have the desired results.
\end{proof}
\begin{proof}[{\bf Proof of Theorem \ref{la4}}]
In view of (\ref{2a5}) and using the inequality $x^2+y^2\geq (x+y)^2/2$, we have
\bea\label{2a8} &&\sum_{k=1}^p \sum_{n=2}^\infty \left((n+k-1)^2+(k-1)^2\right)\left(|a_{n,k}|+|b_{n,k}|\right)^2\nonumber\\
&\leq &(K^2+1)\lambda^2+2K\sqrt{K'}\lambda+K'-\left(k^2+(k-1)^2\right)\left(|a_{1,k}|+|b_{1,k}|\right)^2\nonumber\\
&\leq& (K^2+1)\lambda^2+2K\sqrt{K'}\lambda+K'-\left(|a_{1,1}|+|b_{1,1}|\right)^2.\eea
Furthermore, since $\lambda_F(0)=\left||a_{1,1}|-|b_{1,1}|\right|=1$ and using the inequality $|x|+|y|\geq |x+y|\geq \left||x|-|y|\right|$ for any real $x,y$, from (\ref{2a8}), we have
{\small \beas \sum_{k=1}^p \sum_{n=2}^\infty \left((n+k-1)^2+(k-1)^2\right)\left(|a_{n,k}|+|b_{n,k}|\right)^2
\leq (K^2+1)\lambda^2+2K\sqrt{K'}\lambda+K'-1.\eeas
Thus, for $k=1$, we have
\beas  \sum_{n=2}^\infty n^2\left(|a_{n,k}|+|b_{n,k}|\right)^2\leq (K^2+1)\lambda^2+2K\sqrt{K'}\lambda+K'-1,\eeas}
from which, we have
\beas |a_{n,k}|+|b_{n,k}|\leq \frac{\sqrt{(K^2+1)\lambda^2+2K\sqrt{K'}\lambda+K'-1}}{n}\;\;\text{for all}\;\;n\geq2,k=1.\eeas
In view of (\ref{2a7}), we have the proof of the remaining two cases. This completes the proof.
\end{proof}
\begin{proof}[{\bf Proof of Corollary \ref{c2}}]
By substituting $K'=0$ in \textrm{Theorem \ref{la4}}, we have the desired results.
\end{proof}
\begin{proof}[{\bf Proof of Theorem \ref{la5}}]
Since $J_F(0)=1$, we have $\Lambda_F^2(0)\leq K+K'$. Thus,
\beas \lambda_F(0)=\frac{J_F(0)}{\Lambda_F(0)}\geq \frac{1}{\sqrt{K+K'}}.\eeas Hence, it follows from (\ref{2a8}) that
\beas &&\sum_{k=1}^p \sum_{n=2}^\infty \left((n+k-1)^2+(k-1)^2\right)\left(|a_{n,k}|+|b_{n,k}|\right)^2
\nonumber\\[2mm]&\leq& (K^2+1)\lambda^2+2K\sqrt{K'}\lambda+K'-\frac{1}{K+K'}.\eeas
Thus, for $k=1$, we have
\beas  \sum_{n=2}^\infty n^2\left(|a_{n,k}|+|b_{n,k}|\right)^2\leq (K^2+1)\lambda^2+2K\sqrt{K'}\lambda+K'-\frac{1}{K+K'},\eeas
from which, we have
\beas |a_{n,k}|+|b_{n,k}|\leq \frac{\sqrt{(K^2+1)\lambda^2+2K\sqrt{K'}\lambda+K'-\frac{1}{K+K'}}}{n}\;\;\text{for all}\;\;n\geq2,k=1.\eeas
In view of (\ref{2a7}), we have the proof of the remaining two cases. This completes the proof.
\end{proof}
\begin{proof}[{\bf Proof of Corollary \ref{c3}}]
By substituting $K'=0$ in \textrm{Theorem \ref{la5}}, we have the desired results.
\end{proof}
\begin{proof}[{\bf Proof of Theorem \ref{th1}}]
Differentiating $F(z)$ partially with respect to $z$ and $\ol{z}$, respectively, we have
\beas F_z(z)&=&\left(G_p\right)_z(z)+\sum_{k=2}^p (k-1)z^{k-2}\ol{z}^{k-1}G_{p-k+1}(z)+\sum_{k=2}^p |z|^{2(k-1)}\left(G_{p-k+1}\right)_z(z),\\[2mm]
F_{\ol{z}}(z)&=&\left(G_p\right)_{\ol{z}}(z)+\sum_{k=2}^p (k-1)z^{k-1}\ol{z}^{k-2}G_{p-k+1}(z)+\sum_{k=2}^p |z|^{2(k-1)}\left(G_{p-k+1}\right)_{\ol{z}}(z).\eeas
To prove the univalence of $F$ in $\D_{r_3}$, let us assume $z_1,z_2\in\D_r$ with $z_1\neq z_2$ and $r\in (0,1)$. Then for $[z_1,z_2]$, the line segment joining $z_1$ to $z_2$, we have
\beas \left|F(z_2)-F(z_1)\right| =\left|\int_{[z_1,z_2]}F_z(z)dz+F_{\ol{z}}(z)d\ol{z}\right|\geq I_1-I_2-I_3-I_4,\eeas
where
\beas  I_1&=&\left|\int_{[z_1,z_2]}(G_p)_z(0)dz+(G_p)_{\ol{z}}(0)d\ol{z}\right|,\\[2mm]
I_2&=&\left|\int_{[z_1,z_2]}\left((G_p)_z(z)-(G_p)_z(0)\right)dz+\left((G_p)_{\ol{z}}(z)-(G_p)_{\ol{z}}(0)\right)d{\ol{z}}\right|,\\[2mm]
I_3&=& \left|\sum_{k=2}^p\int_{[z_1,z_2]}|z|^{2(k-1)}\left((G_{p-k+1})_z(z)dz+(G_{p-k+1})_{\ol{z}}(z)d\ol{z}\right)\right|,\\[2mm]
I_4&=&\left|\sum_{k=2}^p\int_{[z_1,z_2]}(k-1)G_{p-k+1}(z)\left(z^{k-2}\ol{z}^{k-1}dz+z^{k-1}\ol{z}^{k-2}d\ol{z}\right)\right|.\eeas
Firstly, since $\lambda_F(0)=1$, we have
\beas I_1\geq \int_{[z_1,z_2]}\lambda_{G_p}(0)|dz|=\int_{[z_1,z_2]}\left||a_{1,1}|-|b_{1,1}|\right||dz|=\int_{[z_1,z_2]}\lambda_F(0)|dz|=|z_2-z_1|.\eeas
Next, in view of Theorem \ref{la4}, we have the following inequalities for $I_2,I_3$ and $I_4$.
\beas I_2&=&\left|\int_{[z_1,z_2]}\sum_{n=2}^\infty n\left(a_{n,1}z^{n-1}dz+\ol{b_{n,1}z^{n-1}}d\ol{z}\right)\right|\\[2mm]
&\leq& |z_2-z_1|\sum_{n=2}^\infty n(|a_{n,1}|+|b_{n,1}|)r^{n-1}\\[2mm]
&\leq & |z_2-z_1|\sqrt{(K^2+1)\lambda^2+2K\sqrt{K'}\lambda+K'-1}\cdot \frac{r}{1-r},\eeas
\beas I_3&=&\left|\sum_{k=2}^p\int_{[z_1,z_2]}|z|^{2(k-1)}\sum_{n=1}^\infty n\left(a_{n,k}z^{n-1}dz+\ol{b_{n,k}z^{n-1}}d\ol{z}\right)\right|\\[2mm]
&\leq& |z_2-z_1|\sum_{k=2}^pr^{2(k-1)}\sum_{n=1}^\infty n(|a_{n,k}|+|b_{n,k}|)r^{n-1}\\[2mm]
&=&|z_2-z_1|\sum_{k=2}^pr^{2(k-1)}\left(|a_{1,k}|+|b_{1,k}|+\sum_{n=2}^\infty n(|a_{n,k}|+|b_{n,k}|)r^{n-1}\right)\\[2mm]
&\leq&|z_2-z_1|\sqrt{(K^2+1)\lambda^2+2K\sqrt{K'}\lambda+K'-1}\\[2mm]&&\sum_{k=2}^pr^{2(k-1)}\left(\frac{1}{\sqrt{5}}+\frac{1}{\sqrt{10}}\frac{2r-r^2}{(1-r)^2}\right)\eeas
and
\beas I_4&=& \left|\sum_{k=2}^p\int_{[z_1,z_2]}(k-1)|z|^{2(k-2)}\sum_{n=1}^\infty \left(a_{n,k} z^n+\ol{b_{n,k} z^n}\right)\left(\ol{z}dz+zd\ol{z}\right)\right|\\[2mm]
&\leq & \left|\sum_{k=2}^p\int_{[z_1,z_2]}(k-1)|z|^{2(k-2)}\sum_{n=1}^\infty 2\left(|a_{n,k}|+|b_{n,k}|\right)|z|^{n+1}|dz|\right|\\[2mm]
&\leq & 2|z_2-z_1|\sum_{k=2}^p (k-1)r^{2(k-2)}\sum_{n=1}^\infty \left(|a_{n,k}|+|b_{n,k}|\right)r^{n+1}\\[2mm]
&=&2|z_2-z_1|\sum_{k=2}^p (k-1)r^{2(k-2)}\left(\left(|a_{1,k}|+|b_{1,k}|\right)r^2+\sum_{n=2}^\infty \left(|a_{n,k}|+|b_{n,k}|\right)r^{n+1}\right)\\[2mm]
&\leq & 2|z_2-z_1|\sqrt{(K^2+1)\lambda^2+2K\sqrt{K'}\lambda+K'-1}\\[2mm]&&\sum_{k=2}^p (k-1)r^{2(k-1)}\left(\frac{1}{\sqrt{5}}+\frac{r}{\sqrt{10}(1-r)}\right) .\eeas
Thus, we have
\beas &&|F(z_2)-F(z_1)|\geq \psi (r) |z_2-z_1|,\eeas
where
{\small \beas \psi(r)
&=&1-\sqrt{(K^2+1)\lambda^2+2K\sqrt{K'}\lambda+K'-1}\Bigg( \frac{r}{1-r}\nonumber\\[2mm]
&&\left.+\sum_{k=2}^pr^{2(k-1)}\left(\frac{1}{\sqrt{5}}+\frac{1}{\sqrt{10}}\frac{2r-r^2}{(1-r)^2}\right)+2\sum_{k=2}^p (k-1)r^{2(k-1)}\left(\frac{1}{\sqrt{5}}+\frac{r}{\sqrt{10}(1-r)}\right)
\right)\\[2mm]
&=&1-\sum_{n=1}^\infty \psi_nr^n.\eeas}
Since $(K^2+1)\lambda^2+2K\sqrt{K'}\lambda+K'>1$, $\psi_n$ are positive for each $n\geq 1$. Thus, $\psi(r)$ is strictly decreasing in $[0,1)$. Furthermore,
\beas \lim_{r\to 0^+} \psi(r)=1, \lim_{r\to 1^-}\psi(r)=-\infty.\eeas
Thus, there exists a unique $r_3\in(0,1)$ satisfying $\psi(r_1)=0$, which shows that $F$ is univalent in the disk $D_{r_3}$.\\\par
Next, for any point $z=r_3e^{i\theta}$ on $\partial \D_{r_3}$, from Theorem (\ref{la4}), we have
\beas |F(z)|&=& \left|\sum_{k=1}^p|z|^{2(k-1)}G_{p-k+1}(z)\right|\\[2mm]
&=&\left|\sum_{k=1}^p |z|^{2(k-1)}\sum_{n=1}^\infty (a_{n,k}z^n+\ol{b_{n,k}z^n})\right|\\[2mm]
&\geq & \lambda_F(0)r_3-\sum_{n=2}^\infty (|a_{n,1}|+|b_{n,1}|)r_3^n\\[2mm]
&&-\sum_{k=2}^{p}r_3^{2(k-1)}\left((|a_{1,k}|+|b_{1,k}|)r_3+\sum_{n=2}^\infty (|a_{n,k}|+|b_{n,k}|)r_3^n\right)\\[2mm]
&\geq& r_3+\sqrt{(K^2+1)\lambda^2+2K\sqrt{K'}\lambda+K'-1}\Bigg(\ln(1-r_3)+r_3\\[2mm]
&&\left.-\sum_{k=2}^{p}r_3^{2(k-1)}\left(\frac{r_3}{\sqrt{5}}+\frac{r_3^2}{\sqrt{10}(1-r_3)}\right)\right).\eeas
This completes the proof.
\end{proof}
\begin{proof}[{\bf Proof of Corollary \ref{c4}}]
By substituting $K'=0$ in Theorem (\ref{th1}), we have the desired results. 
\end{proof}
\begin{proof}[{\bf Proof of Theorem \ref{th2}}]
To prove the univalence of $F$ in $\D_{r_5}$, let us assume $z_1,z_2\in\D_r$ with $z_1\neq z_2$ and $r\in (0,1)$. Then by Theorem \ref{la5} and the proof of Theorem \ref{th1}, we have
\beas |F(z_2)-F(z_1)|\geq \xi(r) |z_2-z_1|,\eeas
where
{\small \beas \xi(r)&=& \frac{1}{\sqrt{K+K'}}-\sqrt{(K^2+1)\lambda^2+2K\sqrt{K'}\lambda+K'-\frac{1}{K+K'}}\Bigg( \frac{r}{1-r}\nonumber\\[2mm]&&+\sum_{k=2}^pr^{2(k-1)}\left(\frac{1}{\sqrt{5}}+\frac{1}{\sqrt{10}}\frac{2r-r^2}{(1-r)^2}\right)+2\sum_{k=2}^p (k-1)r^{2(k-1)}\left(\frac{1}{\sqrt{5}}+\frac{r}{\sqrt{10}(1-r)}\right)
\Bigg)\nonumber\\[2mm]
&=&1-\sum_{n=1}^\infty \xi_nr^n.\eeas}
Since $(K^2+1)\lambda^2+2K\sqrt{K'}\lambda+K'>1/(K+K')$, $\xi_n$ are positive for each $n\geq 1$. Thus, $\xi(r)$ is strictly decreasing in $[0,1)$. Furthermore,
\beas \lim_{r\to 0^+} \xi(r)=\frac{1}{\sqrt{K+K'}}, \lim_{r\to 1^-}\xi(r)=-\infty.\eeas
Thus, there exists a unique $r_5\in(0,1)$ satisfying $\xi(r_5)=0$, which shows that $F$ is univalent in the disk $D_{r_5}$.\\\par
\indent Next, for any point $z=r_5e^{i\theta}$ on $\partial \D_{r_5}$, using Theorem (\ref{la5}) and the proof of Theorem \ref{th1}, we have
\beas |F(z)|&\geq & \lambda_F(0)r_5-\sum_{n=2}^\infty (|a_{n,1}|+|b_{n,1}|)r_5^n\\[2mm]
&&-\sum_{k=2}^{p}r_5^{2(k-1)}\left((|a_{1,k}|+|b_{1,k}|)r_5+\sum_{n=2}^\infty (|a_{n,k}|+|b_{n,k}|)r_5^n\right)\\[2mm]
&\geq& \frac{r_5}{\sqrt{K+K'}}+\sqrt{(K^2+1)\lambda^2+2K\sqrt{K'}\lambda+K'-\frac{1}{K+K'}}\Bigg(r_5\\[2mm]
&&\left.+\ln(1-r_5)-\sum_{k=2}^{p}r_5^{2(k-1)}\left(\frac{r_5}{\sqrt{5}}+\frac{r_5^2}{\sqrt{10}(1-r_5)}\right)\right).\eeas
This completes the proof.
\end{proof}
\begin{proof}[{\bf Proof of Corollary \ref{c5}}]
By substituting $K'=0$ in Theorem (\ref{th2}), we have the desired results. 
\end{proof}
\section*{Declarations}
\noindent{\bf Acknowledgement:} The second Author is supported by University Grants Commission (IN) fellowship (No. NBCFDC/CSIR-UGC-DECEMBER-2023).\\[2mm]
{\bf Conflict of Interest:} The authors declare that there are no conflicts of interest regarding the publication of this paper.\\[2mm]
{\bf Availability of data and materials:} Not applicable

\end{document}